\documentclass[12pt,reqno]{amsart}
\usepackage{enumitem}
\usepackage{graphicx, tikz, tkz-graph}
\usepackage{wrapfig, float,  caption, subcaption}
\usepackage{amsmath,amssymb}
\usepackage{caption}
\usetikzlibrary{graphs}
\usetikzlibrary{graphs.standard, quotes, arrows, positioning, calc, arrows.meta}
\usepackage{pgfplots}
\usepackage{xcolor}
\pgfplotsset{compat=newest}
\usepgfplotslibrary{polar}
\usepgflibrary{shapes.geometric}
\usetikzlibrary{calc}
\pgfplotsset{my style/.append style={
               axis x line=middle, 
               axis y line= middle,
               xlabel={$x$}, 
               ylabel={$y$},
               legend pos = south east,
               label style={font=\footnotesize},
               tick label style={font=\tiny}}}
\usepgfplotslibrary{fillbetween}
\usetikzlibrary{patterns}
\usepackage{harpoon}
\newtheorem{theorem}{Theorem}
\newtheorem{lemma}[theorem]{Lemma}
\newtheorem{corollary}[theorem]{Corollary}
\newtheorem{proposition}[theorem]{Proposition}

\newtheorem{sublemma}{}[theorem]

\newtheorem*{theorem*}{Theorem}
\theoremstyle{definition}

\theoremstyle{remark}

\definecolor{mygray}{RGB}{105, 105, 105}

\numberwithin{equation}{section}
\newcommand{\del}{\backslash}

\newcommand{\I}{\mathcal{I}}

\newcommand{\cl}{\text{cl}}
\newcommand{\sgn}{\text{sgn}}

\makeatletter
\pgfdeclarepatternformonly[\GridSize]{MyGrid}{\pgfqpoint{-1pt}{-1pt}}{\pgfqpoint{\GridSize}{\GridSize}}{\pgfqpoint{\GridSize}{\GridSize}}
{
     \pgfsetcolor{\tikz@pattern@color}
     \pgfsetlinewidth{0.3pt}
     \pgfpathmoveto{\pgfqpoint{0pt}{0pt}}
     \pgfpathlineto{\pgfqpoint{0pt}{\GridSize + 0.1pt}}
     \pgfpathmoveto{\pgfqpoint{0pt}{0pt}}
     \pgfpathlineto{\pgfqpoint{\GridSize + 0.1pt}{0pt}}
     \pgfusepath{stroke}
}
\makeatother

\newdimen\GridSize
\tikzset{
    GridSize/.code={\GridSize=#1},
    GridSize=6pt
}

\makeatletter
\pgfdeclarepatternformonly[\LineSpace]{my north west lines}{\pgfqpoint{-1pt}{-1pt}}{\pgfqpoint{\LineSpace}{\LineSpace}}{\pgfqpoint{\LineSpace}{\LineSpace}}
{
    \pgfsetcolor{\tikz@pattern@color}
    \pgfsetlinewidth{0.4pt}
    \pgfpathmoveto{\pgfqpoint{0pt}{\LineSpace}}
    \pgfpathlineto{\pgfqpoint{\LineSpace + 0.1pt}{-0.1pt}}
    \pgfusepath{stroke}
}
\makeatother
\newdimen\LineSpace
\tikzset{
    line space/.code={\LineSpace=#1},
    line space=6.5pt
}
\begin{document}

\title{Weak maps and the Tutte polynomial}

\author{Christine Cho and James Oxley}
\address{Mathematics Department\\
  Louisiana State University\\
  Baton Rouge, Louisiana}
\email{ccho3@lsu.edu}

\date{\today}

\begin{abstract}
    Let $M$ and $N$ be matroids such that $N$ is the image of $M$ under a rank-preserving weak map. 
    Generalizing results of Lucas, we prove that, for $x$ and $y$ positive, $T(M;x,y)\geq T(N;x,y)$ if and only if $x+y\geq xy$ or $M\cong N$. 
    We give a number of consequences of this result. 
\end{abstract}

\maketitle
\section{Introduction}

Terminology and notation used here will follow \cite{ox} unless otherwise stated.
Given two rank-$r$ matroids $M$ and $N$, a bijective map from $E(M)$ to $E(N)$ is a \textit{rank-preserving weak map} if every basis of $N$ is the image of a basis of $M$. 
We write $M\xrightarrow{rp} N$ if $N$ is a rank-preserving weak-map image of $M$. 

The following theorem of Lucas \cite{lucas} shows that the numbers of bases, independent sets, and spanning sets of $M$ are greater than the corresponding numbers for $N$ if $M\xrightarrow{rp} N$.
Note that $T(M;1,1)$, $T(M;2,1)$, and $T(M;1,2)$ count the numbers of bases, independent sets, and spanning sets of $M$, respectively, where $T(M;x,y)$ is the Tutte polynomial of $M$. 

    \begin{theorem}\label{lucas}
        If $M\not\cong N$ and $M\xrightarrow{rp}N$, then  
        \begin{enumerate}[label=(\roman*)]
           \item $T(M;1,1)>T(N;1,1)$;
           \item $T(M;2,1)>T(N;2,1)$;
           \item $T(M;1,2)>T(N;1,2)$; 
           \item $T(M;x,0)\geq T(N;x,0)$ for all $x>0$ unless $M$ has a loop;
           \item $T(M;0,y)\geq T(N;0,y)$ for all $y>0$ unless $M$ has a coloop. 
       \end{enumerate} 
    \end{theorem}

The main result of the paper is the following generalization of the last theorem. 

    \begin{theorem}\label{main}
        Let $x$ and $y$ be positive real numbers. Let $M$ and $N$ be matroids such that there is a rank-preserving weak map from $M$ to $N$. Then $T(M;x,y)\geq T(N;x,y)$ if and only if $x+y\geq xy$ or $M\cong N$. Moreover, if $M\not\cong N$, then $T(M;x,y)>T(N;x,y)$ if and only if $x+y>xy$.  
    \end{theorem}

Jaeger, Vertigan, and Welsh \cite{jaeger} proved that the problem of evaluating the Tutte polynomial of a graphic matroid at a point $(x,y)$ in the first quadrant of the real plane is \#P-hard unless $x+y=xy$ or $(x,y)=(1,1)$. 
Evidently $x+y=xy$ if and only if $(x,y)$ is a point on the hyperbola $H_1$ defined by the equation $(x-1)(y-1)=1$.  
It is straightforward to prove that $T(M;x,y)=(x-1)^{r(M)}y^{|E|}$ for all $(x,y)\in H_1$. 
Therefore, for any two matroids $M$ and $N$ that have the same rank and the same ground set, $T(M;x,y)=T(N;x,y)$ for all $(x,y)\in H_1$, that is, for all $(x,y)$ for which $x+y=xy$. Theorem \ref{main} is summarized in the following figure. 
\begin{figure}[ht]
\begin{center}
\begin{tikzpicture}[domain=0:7]
\begin{axis}[
	axis lines = middle,
    x label style={at={(axis description cs:1.05,0.0667)},anchor=center},
	xlabel = {$x$},
    y label style={at={(axis description cs:0.0667,1.05)},anchor=center},
	ylabel = {$y$},
    xtick={2},
    ytick={2},
	xmin=-.5, xmax=7,
	ymin=-0.5, ymax=7, 
    clip mode=individual]
    \addlegendimage{area legend, pattern=my north west lines} 
    \addlegendimage{stealth-stealth, thick,black}
    \addlegendimage{area legend, pattern = grid}      
    
\addplot [name path = A,
	stealth-stealth,
	domain = 1.1:7,
	samples = 100,color=black,very thick] {x/(x-1)} 
	node [very near end, above right] {};
    
\addplot [color=black, name path = B,
	-stealth,
	domain = 0:7] {0} 
    node [pos=1,above left] {};
\addplot [color=white, name path = C,
	domain = 0:1.13] {10};
\addplot [color=white, name path = D,
	domain = 1.11:10] {10};

\node[circle,fill,inner sep=2pt] at (axis cs:2,2) {};

\draw [-latex, color=black,-stealth] (axis cs:{0,0}) -- (axis cs:{0,7}) node[below right] {} ;

  \addplot[pattern=MyGrid, pattern color=black!50]fill between[of=D and A, soft clip={domain=1.1:10}]; 
  \addplot [pattern=my north west lines, pattern color = black!50] fill between [of = A and B, soft clip={domain=1.1:7}];
  \addplot [pattern=my north west lines, pattern color = black!50] fill between [of = B and C, soft clip={domain=0.01:1.12}];
  \addlegendentry{\footnotesize$T(M;x,y) > T(N;x,y)$};
  \addlegendentry{\footnotesize$T(M;x,y) = T(N;x,y)$};
  \addlegendentry{\footnotesize$T(M;x,y) < T(N;x,y)$};

  \node[anchor=west,overlay] () at (axis cs:7,1.1667) {$x+y =  xy$};
  \end{axis}
\end{tikzpicture}
\end{center}
\caption{A summary of Theorem \ref{main}.}
\end{figure}
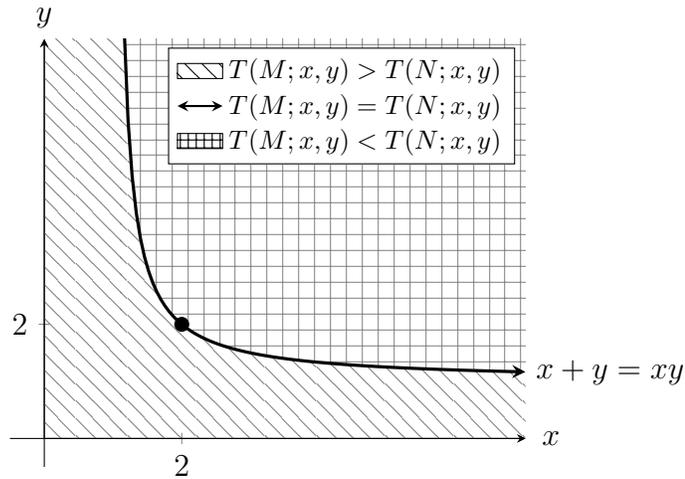

\begin{figure}
    \centering
    \begin{tikzpicture}[rotate={90}]
    \filldraw (0,0) circle (4pt) node[below right=1pt] {$x_1$};
    \filldraw (0.5,1) circle (4pt) node[below=4pt] {$x_2$};
    \filldraw (1.5,1) circle (4pt) node[above=4pt] {$x_4$};
    \filldraw (.87,1.95) circle (4pt) node[below left= 2pt] {$x_3$};
    \filldraw (1.13,1.95) circle (4pt) node[above left=1pt] {$x_{3}'$};
    \filldraw (2,0) circle (4pt) node[above right=1pt] {$x_5$};
    \filldraw (1,-.75) circle (4pt) node[right=3pt] {$x_6$};

    \draw[thick] (0,0)--(1,2);
    \draw[thick] (1,2)--(2,0);
\end{tikzpicture}
\caption{A rank-3 matroid $M$.}\label{qsix}
\end{figure}
Let $f$ and $g$ be distinct elements in a matroid $M$. The element $f$ is \textit{freer} than the element $g$ if $g$ is contained in the closure of every circuit containing $f$. For example, in the rank-3 matroid $M$ for which a geometric representation is shown in Figure \ref{qsix}, the element $x_6$ is \textit{free}, that is, every circuit containing $x_6$ is spanning. Hence $x_6$ is freer than every element in $E(M)-\{x_6\}$. The element $x_1$ is freer than both $x_3$ and $x_3'$. Since $\{x_3,x_3'\}$ is a parallel class, it is straightforward to see that $x_3$ is freer than $x_3'$, and $x_3'$ is freer than $x_3$. Similarly, $x_1$ is freer than $x_2$, and $x_2$ is freer than $x_1$. By contrast, although $x_1$ and $x_5$ are symmetric in $M$, neither is freer than the other.

As a consequence of Theorem \ref{main}, we deduce that if $f$ is freer than $g$ in a matroid $M$, then the numbers of bases, circuits, and hyperplanes of $M$ containing $f$ are at least as large as the corresponding numbers of sets containing $g$. 
The next section presents some preliminaries. 
The main result is proved in Section 3.
The last section contains consequences of the main theorem. 

\section{Preliminaries}
The \textit{nullity} of a matroid $M$ is equal to $|E(M)|-r(M)$. 
For matroids $M$ and $N$, a bijection $\varphi:E(M) \to E(N)$ is a \textit{weak map} if $\varphi^{-1}(I)\in \I(M)$ whenever $I \in \I(N)$.
If $r(M)=r(N)$, then $\varphi$ is a \textit{rank-preserving weak map} from $M$ to $N$.
Although it is not required that a weak map be bijective, we will only consider bijective weak maps. 
Such maps have the following attractive property (see, for example, \cite[Corollary~7.3.13]{ox}).     

\begin{lemma}
   If $\varphi: M\to N$ is a rank-preserving weak map from $M$ to $N$, then $\varphi$ is a rank-preserving weak map from $M^*$ to $N^*$. 
\end{lemma}

 For a matroid $M$ with ground set $E$, the \textit{Tutte polynomial} $T(M;x,y)$ of $M$ is defined by 
    \[
        T(M;x,y)=\displaystyle\sum_{A\subseteq E} (x-1)^{r(M)-r(A)}(y-1)^{|A|-r(A)}.
    \]
    It is well known that $T(M^*;x,y)=T(M;y,x)$ for a matroid $M$ and its dual $M^*$. In-depth accounts of the Tutte polynomial and its applications can be found in \cite{bry} and \cite{handbook}. In \cite[Exercise 6.10(b)]{bry}, it is noted that $T(M;x,y)=T(N;x,y)-xy+x+y$ if $M$ is obtained from $N$ by relaxing a circuit-hyperplane. Since relaxation is an example of a rank-preserving weak map, this adds to the plausibility of the main result.   

     Before proving the main result in general, we prove it in the specific case when $N$ is comprised solely of loops and coloops. Observe that if there is a rank-preserving weak map from a matroid $M$ to a matroid $N$, then every coloop of $M$ is a coloop of $N $, while every loop of $M$ is a loop of $N$. The proof of the following lemma uses the sign function $\text{sgn} : \mathbb{R} \to \{-1,0,1\}$ where 
        \[
           \text{sgn}(x)=\begin{cases}
               -1 & \text{ if } x<0,\\  
                0 & \text{ if } x=0,\\  
                1 & \text{ if } x>0.\\  
           \end{cases}
        \] 
\begin{lemma}\label{lc}
    Let $x>0$ and $y>0$. Let $M$ be a matroid with rank $k$, nullity $m$, and $|E|\geq 2$. Then $T(M;x,y)\geq x^ky^{m}$ if and only if $x+y\geq xy$ or $M\cong U_{k,k}\oplus U_{0,m}$. Moreover, if $M\not\cong U_{k,k}\oplus U_{0,m}$, then $T(M;x,y)>x^ky^m$ if and only if $x+y>xy$. 
\end{lemma}
\begin{proof}
     Suppose $M\not\cong U_{k,k}\oplus U_{0,m}$. We argue by induction on $|E|$ that $\text{sgn}(T(M)-x^ky^m)= \text{sgn}(x+y-xy)$, where we have abbreviated $T(M;x,y)$ as $T(M)$. 
    
    If $|E|=2$, then $M\cong U_{1,2}$ and $T(M;x,y)=x+y$. Therefore $\text{sgn}(T(M)-xy)=\text{sgn}(x+y-xy)$, and the result holds. Assume that the result holds for $|E|<n$ and let $|E|=n\geq 3$. Since $M\not\cong U_{k,k}\oplus U_{0,m}$, there is an element $e$ of $M$ that is neither a loop nor a coloop. 
    Thus $T(M)=T(M\del e)+T(M/e)$.  

    Suppose $M\del e\cong U_{k,k}\oplus U_{0,m-1}$. Then $M\cong U_{s,s+1}\oplus U_{k-s,k-s}\oplus U_{0,m-1}$ for some $s$ with $1\leq s\leq k$. If $s=1$, then $M/e\cong U_{k-1,k-1}\oplus U_{0,m}$. Thus $T(M\del e)=x^{k}y^{m-1}$ and $T(M/e)=x^{k-1}y^m$.
    Since $x>0$ and $y>0$,   
        \begin{align*}
            \sgn(T(M)-x^ky^m)&=\sgn(x^ky^{m-1}+x^{k-1}y^m-x^ky^m)\\
                                   &=\sgn(x^{k-1}y^{m-1}(x+y-xy))=\sgn(x+y-xy),
    \end{align*}
   as desired.  

    Suppose $s\geq 2$. Then $M/e\not\cong U_{k-1,k-1}\oplus U_{0,m}$ and, by the induction assumption,  
        \begin{align*}
        \sgn(T(M/e)-x^{k-1}y^{m})
         &=\sgn(x+y-xy).
        \end{align*}
    Now, by the deletion-contraction formula for $T(M)$, 
     \begin{align*}
            \sgn(T(M)-x^ky^m)&=\text{sgn}(T(M/e)+x^ky^{m-1}-x^ky^m)\\
            &=\sgn(T(M/e)-x^{k-1}y^{m-1}(xy-x)).
        \end{align*}
    If $\sgn(x+y-xy)=1$, then $y>xy-x$ and $T(M/e)>x^{k-1}y^{m-1}(y)$. Hence $\text{sgn}(T(M/e)-x^{k-1}y^{m-1}(xy-x))=1=\text{sgn}(T(M)-x^ky^m)$. Applying analogous arguments to the remaining two cases, it follows that $\text{sgn}(T(M)-x^ky^m)=\text{sgn}(x+y-xy)$, as desired.

    We may now assume that $M\del e\not\cong U_{k,k}\oplus U_{0,m-1}$. By duality, we may also assume that $M/e\not\cong U_{k-1,k-1}\oplus U_{0,m}$. By the induction assumption, 
    \[
        \text{sgn}(T(M\del e)-x^ky^{m-1})=\text{sgn}(T(M/e)-x^{k-1}y^m)=\text{sgn}(x+y-xy),
    \]
    so $\text{sgn}(T(M\del e)-x^ky^{m-1}+T(M/e)-x^{k-1}y^m)=\text{sgn}(x+y-xy).$
    It follows that  
        \begin{align*}
            \text{sgn}(x+y-xy)
            &=\text{sgn}(T(M\del e)+T(M/e)-x^ky^{m-1}-x^{k-1}y^m)\\
            &=\text{sgn}(T(M)-x^{k-1}y^{m-1}(x+y))\\
            &=\text{sgn}(T(M)-x^{k-1}y^{m-1}(xy))\\
            &=\text{sgn}(T(M)-x^ky^m),
        \end{align*} 
    where the third equality follows by checking each of the three possibilities for $\sgn(x+y-xy).$ We conclude that the lemma holds. 
\end{proof}

\section{Proof of the Main Theorem}
The following argument follows the same general structure as the proof of Lemma \ref{lc}. We again use the abbreviation $T(M)$ in place of $T(M;x,y)$.
\begin{proof}[Proof of Theorem \ref{main}.]
    It suffices to prove the result when $M$ and $N$ have a common ground set $E$ and the rank-preserving weak map from $M$ to $N$ is the identity map. 
    We will argue by induction on $|E|$ that $\text{sgn}(T(M)-T(N))=\text{sgn}(x+y-xy)$ whenever $M\neq N$. 
    
    Let $|E|=2$. 
    Since $U_{1,2}$ and $U_{1,1}\oplus U_{0,1}$ are the only 2-element matroids of equal rank, we must have that $M\cong U_{1,2}$ and $N\cong U_{1,1}\oplus U_{0,1}$. 
    As \mbox{$T(U_{1,2};x,y)=x+y$} and $T(U_{1,1}\oplus U_{0,1};x,y)=xy$, we see that the result holds for $|E|=2$.

    Assume the result holds for $|E|<n$ and let $|E|=n\geq 3$. Take $e\in E$.
    If $e$ is a coloop of $M$, then $e$ is a coloop of $N$, so $T(M)=xT(M\del e)$ and $T(N)=xT(N\del e)$. 
    Therefore, as $x>0$, 
    \[\text{sgn}(T(M)-T(N))=\text{sgn}(T(M\del e)-T(N\del e))=\text{sgn}(x+y-xy).\] 
    Applying a similar argument to the dual, we see that the assertion holds if $M$ has a loop. 

    Suppose $e$ is not a loop or a coloop of $N$.
    Then  
        \begin{align}\label{eq1}
            \text{sgn}(T(M)-T(N))&=\text{sgn}(T(M\del e)+T(M/e)-T(N\del e)-T(N/e))\nonumber\\
                                 &=\text{sgn}(T(M\del e)-T(N\del e)\nonumber\\
                                 &\hspace{2.8cm}+T(M/ e)-T(N/e)). 
        \end{align}
    Since $M\neq N$, we have that $M\del e\neq N\del e$ or $M/e\neq N/e$. 

    Suppose that $M\del e\neq N\del e$ and $M/e\neq N/e$. Then, by the induction assumption,  
    \[
        \text{sgn}(T(M\del e)-T(N\del e))=\text{sgn}(x+y-xy)=\text{sgn}(T(M/e)-T(N/e)).
    \]
    Thus $\text{sgn}(T(M\del e)-T(N\del e)+T(M/e)-T(N/e))=\text{sgn}(x+y-xy)$, that is, $\text{sgn}(T(M)-T(N))=\text{sgn}(x+y-xy)$.  

Now suppose that $M\del e=N\del e$ or $M/e=N/e$. Then, by (\ref{eq1}),  
    \begin{align*}
        \text{sgn}(T(M)-T(N))&=\begin{cases}
                                     \text{sgn}(T(M\del e)-T(N\del e)) &\text{ if } M/e=N/e,\\
                                     \text{sgn}(T(M/e)-T(N/e)) &\text{ if } M\del e=N\del e,\\
                                 \end{cases}\\
                             &= \text{sgn}(x+y-xy)
    \end{align*}
    where the last step follows by the induction assumption. 

    Finally, if every element of $N$ is a loop or a coloop, then $N\cong U_{k,k}\oplus U_{0,m}$, so $T(N;x,y)=x^{k}y^{m}$.
    Since $M\xrightarrow{rp}N$, the matroid $M$ has rank $k$ and nullity $m$. 
    The result follows immediately from Lemma \ref{lc}.  
\end{proof}

\section{Consequences}
A flat $F$ of a matroid $M$ is \textit{cyclic} if $F$ is a union of circuits. Given distinct elements $f$ and $g$ of a matroid $M$, it is well known that $f$ is freer than $g$ if $g$ is contained in every cyclic flat containing $f$. 
It is worth noting that, if $f$ is a coloop of $M$, then $f$ is vacuously freer than $g$ for all $g\in E(M)-\{f\}$. 
Likewise, if $g$ is a loop of $M$, then $f$ is freer than $g$ for all $f\in E(M)-\{g\}$. 
Consequently, our discussion of relative freedom is primarily concerned with elements of $M$ that are neither loops nor coloops. 

Duke showed that relative freedom extends nicely to both duals and minors. If $f$ is freer than $g$ in $M$, then $g$ is freer than $f$ in $M^*$. Moreover, $f$ is freer than $g$ in $M\del X/Y$ for all disjoint subsets $X$ and $Y$ of $E(M)$.  

This section explores the notion of relative freedom of elements of a matroid and its connection to weak maps and the Tutte polynomial. The following result provides our first direct link between relative freedom and rank-preserving weak maps.

Define the map $\varphi_{gf}: E(M/f) \to E(M/g)$ by taking $\varphi_{gf}(g)=f$ and $\varphi_{gf}(e)=e$ for all $e\neq g$. 

    \begin{lemma}\label{freermap}
       If $f$ is freer than $g$ in a matroid $M$ and $g$ is not a loop of $M$, then $\varphi_{gf}$ is a rank-preserving weak map from $M/f$ to $M/g$.  
   \end{lemma}
        \begin{proof}
            Let $I$ be independent in $M/g$. 
            Then $I\cup g$ is independent in $M$.
            Suppose $f\not\in I$. 
            Then $\varphi_{gf}^{-1}(I)=I$. 
            If $I$ is dependent in $M/f$, then $M$ has a circuit $C$ such that $C\subseteq I\cup f$.
            Moreover, $f\in C$ since $I$ is independent in $M$. 
            As $f$ is freer than $g$ in $M$, we see that $g\in \cl_M(C)$.
            Then $I\cup g$ contains a circuit of $M$, a contradiction.
            Therefore $I$ is independent in $M/f$. 
            
            Suppose $f\in I$. 
            Then $f\in I\cup g$ and $I\cup g$ is independent in $M$. 
            Therefore $\varphi_{gf}^{-1}(I)$, which equals $(I\cup g)-f$, is independent in $M/f$. 
        \end{proof}

The next result follows immediately from Theorem \ref{main} and Lemma \ref{freermap}. The straightforward proof is omitted. 

    \begin{corollary}\label{combos}
        Let $x>0$ and $y>0$. If $f$ is freer than $g$ in $M$ and $g$ is not a loop of $M$, then $T(M/f;x,y)\geq T(M/g;x,y)$ if and only if $x+y\geq xy$ or $M/f\cong M/g$. 
    \end{corollary}

    \begin{corollary}
        Let $x>0$ and $y>0$. If $f$ is freer than $g$ in $M$ and $g$ is not a coloop of $M$, then $T(M\del f;x,y)\leq T(M\del g;x,y)$ if and only if $x+y\geq xy$ or $M\del f\cong M\del g$. 
    \end{corollary}

    \begin{proof}
        Since $g$ is freer than $f$ in $M^*$, it follows by Corollary \ref{combos} that $T(M^*/f;y,x)\leq T(M^*/g;y,x)$ if and only if $x+y \geq xy$ or $M^*/f\cong M^*/g$. Thus, by duality, we have $T(M\del f;x,y)\leq T(M\del g;x,y)$ if and only if $x+y\geq xy$ or $M\del f\cong M\del g$. 
    \end{proof}

    The following result lists several consequences of Corollary \ref{combos}. We use $b(M)$, $W_k(M)$, $h(M)$, and $\gamma(M)$ to represent the numbers of bases, rank-$k$ flats, hyperplanes, and circuits of $M$, respectively. To specify the numbers of such sets containing some element $e$ of $M$, we write, for example, $b(e;M)$ and $W_k(e;M)$. 

    \begin{corollary}\label{cnt}
        If $f$ is freer than $g$ in $M$, then 
        \begin{enumerate}[label=(\roman*)]
            \item $b(f;M)\geq b(g;M)$; 
            \item $W_{k}(f;M)\geq W_{k}(g;M)$ for all $k\geq 0$, provided $g$ is not a loop of $M$; 
            \item $h(f;M)\geq h(g;M)$, provided $g$ is not a loop of $M$;
            \item $\gamma(f;M)\geq \gamma(g;M)$, provided $f$ is not a coloop of $M$.

        \end{enumerate}
    \end{corollary}
    \noindent \textit{Proof.}
    For (i), note that $b(e;M)=b(M/e)$ for $e\in E(M)$ as long as $e$ is not a loop. If $g$ is a loop of $M$, then $b(g;M)=0$, so (i) holds. Assume $g$ is not a loop of $M$. As $f$ is freer than $g$ and $g$ is not a loop, $f$ is not a loop. Thus the map $\varphi_{gf}$ is a rank-preserving weak map from $M/g$ to $M/f$. By Theorem \ref{main}, we have $b(M/f)=T(M/f;1,1)\geq T(M/g;1,1)=b(M/g)$. Thus (i) holds.

     To prove (ii), observe that, when $e$ is not a loop of $M$, a set $X$ is a rank-$k$ flat of $M$ if and only if $X-e$ is a rank-$(k-1)$ flat of $M/e$. 
     Suppose $g$ is not a loop of $M$. 
     Then, since $\varphi_{gf}$ is a rank-preserving weak map from $M/f$ to $M/g$, it follows by \cite[Proposition 9.3.3]{kung}, that $W_{k-1}(M/f)\geq W_{k-1}(M/g)$ for all $k\geq 1$. Thus $W_{k}(f;M)\geq W_{k}(g;M)$ for all $k\geq 1$. Also  $W_{0}(f;M)=0=W_{0}(g;M)$ since neither $f$ nor $g$ is a loop of $M$. Thus (ii) holds. Hence so does (iii).  

    For (iv), observe that $\gamma(e;M)=h(M^*)-h(e;M^*)$. Assume $f$ is not a coloop of $M$. Since $g$ is freer than $f$ in $M^*$, we have, by (iii), that $h(f;M^*)\leq h(g;M^*)$. Therefore 

    $$h(M^*)-h(f;M^*)\geq h(M^*)-h(g;M^*)$$ and (iv) holds.\hfill $\Box$ 

 \qquad

To illustrate (ii) of the last corollary, consider the matroid $M$ in Figure \ref{qsix}, where $x_1$ is freer than $x_3$. 
Evidently, $W_2(x_1;M)=4$ and $W_2(x_3;M)=3$, so $W_2(x_1;M)>W_2(x_3;M)$. 
By contrast, since every cyclic flat containing $x_1$ also contains $x_3$, the number of rank-2 cyclic flats containing $x_3$ is at least the number of rank-2 cyclic flats containing $x_1$. Indeed, the former is 2 and the latter 1. 

Let $\gamma'(e;M)$ be the number of spanning circuits of $M$ containing an element $e$ of $M$.  

    \begin{corollary}
        If $f$ is freer than $g$ in $M$ and $f$ is not a coloop of $M$, then $\gamma'(f;M)\geq \gamma'(g;M)$.     
    \end{corollary}
    
        \begin{proof}
           Take a spanning circuit $D$ of $M$ containing $g$ but not $f$. Then $D=B\cup g$ for some basis $B$ of $M$. Suppose $B \cup f$ is not a circuit of $M$. Then $B\cup f$ properly contains a circuit $C$ of $M$ and $f\in C$. Hence $\cl(C)$ contains $g$. The set $C-f$ spans $C$, so $g \in \cl(C-f)$. Thus $(C-f)\cup g$ is a dependent set that is a proper subset of the circuit $D$, a contradiction.  
        \end{proof}

        {\begin{lemma}\label{restriction}
        The following are equivalent for elements $f$ and $g$ in a matroid $M$. 
        \begin{enumerate}[label=(\roman*)]
            \item $f$ is freer than $g$ in $M$;  
            \item $b(f;N)\geq b(g;N)$ for all restrictions $N$ of $M$ containing $\{f,g\}$. 
        \end{enumerate}
    \end{lemma}}

    \begin{proof}
        Suppose (i) holds. Then $f$ is freer than $g$ in all restrictions $N$ of $M$ containing $\{f,g\}$, so (ii) holds by Corollary \ref{cnt}(i). 
        
        Suppose (ii) holds and suppose $f$ is not freer than $g$. Then $M$ has a cyclic flat $F$ containing $f$ and avoiding $g$. Let $N'=M|(F\cup g)$. Note that $g$ is a coloop of $N'$. 
        Then $b(g;N')=b(N')$. 
        As $b(f;N')\geq b(g;N')$, it follows that $b(f;N')=b(N')$. Thus $f$ is a coloop of $N'$, a contradiction. 
    \end{proof}

    In the remaining results of this paper, we investigate several instances of equality holding between the number of distinguished sets of $M$ containing $f$ and the number of such sets containing $g$. 
    Let $x$ and $y$ be elements of $M$. 
    Then $x$ and $y$ are \textit{clones} in $M$ if and only if the bijection from $E(M)$ to $E(M)$ that interchanges $x$ and $y$ and fixes every other element is an isomorphism.
    It was shown in \cite[Proposition 4.9]{geelen} that $x$ and $y$ are clones if and only if the set of cyclic flats containing $x$ is equal to the set of cyclic flats containing $y$. 
    Thus $x$ and $y$ are clones if and only if $x$ is freer than $y$, and $y$ is freer than $x$ in $M$. As an example, in Figure \ref{qsix}, the elements $x_3$ and $x_3'$ are clones, as are $x_1$ and $x_2$, and $x_4$ and $x_5$, but there are no other pairs of clones in $M$. Two elements that are parallel are clones as are two elements that are in series.

    \begin{theorem}\label{beq}
        Let $f$ be freer than $g$ in $M$. Then $b(f;M)=b(g;M)$ if and only if $f$ and $g$ are clones in $M$.     
    \end{theorem}
        \begin{proof}
            If $f$ and $g$ are clones in $M$, then clearly $b(f;M)=b(g;M)$. 
            To prove the converse, suppose $b(f;M)=b(g;M)$. 
            First assume that $g$ is a loop of $M$.             
            Then $b(g;M)=0=b(f;M)$. 
            Thus $f$ is a loop of $M$. 
            Therefore $f$ and $g$ are clones in $M$.
            Similarly, if $g$ is a coloop of $M$, then $f$ and $g$ are clones in $M$. 
            We may assume that $f$ and $g$ are neither loops nor coloops. 
            Thus $b(f;M)=b(M/f)$ and $b(g;M)=b(M/g)$. 

            Let $|E(M)|\in \{2,3\}$. 
            Since $f$ and $g$ are not loops or coloops in $M$, we have that $M\in \{U_{1,2}, U_{2,3}, U_{1,3},U_{1,2}\oplus U_{0,1}, U_{1,2}\oplus U_{1,1}\}$. It is straightforward to check that, in these cases, $f$ and $g$ are clones. 

            Assume the result holds for $|E(M)|<n$ and let $|E(M)|=n\geq 4$. 
            Suppose $f$ and $g$ are not clones in $M$. 
            Take an element $e\in E(M)-\{f,g\}$.
            If $e$ is a loop or a coloop in $M$, then $b(f;M\del e)=b(f;M)$ and $b(g;M\del e)=b(g;M)$. 
            Thus $b(f;M\del e)=b(g;M\del e)$. 
            By the induction assumption, $f$ and $g$ are clones in $M\del e$.  
            Hence $f$ and $g$ are clones in $M$, a contradiction. 
            Thus $e$ is neither a loop nor a coloop of $M$. 

            Suppose $\{e,f\}$ is a circuit of $M$. 
            Then $\{e,f,g\}$ is contained in a parallel class since $f$ is freer than $g$ and $g$ is not a loop in $M$, so $f$ and $g$ are clones of $M$.
            Thus we may assume that $e$ is not a loop of $M/f$.
            If $\{e,g\}$ is a circuit of $M$, then $e$ is a loop in $M/g$, so $b(M/g)=b(M/g\del e)$. 
            Therefore $b(M/f\del e)+b(M/f/e)=b(M/g\del e)$. 
            By Lemma \ref{restriction}, $b(M/f\del e)\geq b(M/g\del e)$, so $b(M/f)>b(M/g)$, a contradiction.  

            Now $e$ is not a loop or a coloop of $M$, $M/f $, or $M/g$ and it follows that $b(M/f)=b(M/f\del e)+b(M/f/e)$ and $b(M/g)=b(M/g\del e)+b(M/g/e)$.
            By assumption,  
            \[
                b(M/f\del e)+b(M/f/e)=b(M/g\del e)+b(M/g/e).  
            \] 
            By Lemma \ref{restriction}, $b(M/f\del e)\geq b(M/g\del e)$. 
            Thus $b(M/f/e)\leq b(M/g/e)$.
            Since $f$ is freer than $g$ in $M/e$, it follows, by Corollary \ref{cnt}(i), that $b(M/f/e)=b(M/g/e)$.
            Consequently, $b(M/f\del e)=b(M/g\del e)$. 
            Therefore, by the induction assumption, $f$ and $g$ are clones in $M\del e$.

            Since $f$ and $g$ are not clones in $M$, there is a circuit $C$ of $M$ containing $g$ such that $f\not\in \cl_M(C)$.  
            Assume there is an element $e\in E(M)-(C\cup f)$.
            Then $C$ is a circuit of $M\del e$ containing $g$ such that $f\not\in \cl_{M\del e}(C)$. 
            Hence $f$ and $g$ are not clones in $M\del e$, a contradiction. 
            It follows that $C=E(M)-\{f\}$. Thus $r_M(C)=r(M\del f)=r(M)-1$. Therefore $f$ is a coloop of $M$, a contradiction.   
        \end{proof} 

\begin{proposition}\label{hypclones}
    Let $f$ be freer than $g$ in $M$. Let $L$ be obtained from $M$ by deleting every element of $E(M)-\{f,g\}$ that is parallel to $g$. Then $h(f;M)=h(g;M)$ if and only if $f$ and $g$ are clones in $L$. 
\end{proposition}

\begin{proof}
    Suppose $f$ and $g$ are clones in $L$.
    As $h(f;M)=h(f;L)$ and $h(g;M)=h(g;L)$, we have $h(f;M)=h(g;M)$. 
    To prove the converse, suppose that $h(f;M)=h(g;M)$. 
    Let $\mathcal{H}(M;f;\overline g)$ be the set of hyperplanes of $M$ containing $f$ but not $g$. 
    Then $|\mathcal{H}(M;f;\overline g)|=|\mathcal{H}(M;g;\overline f)|$.
    Hence $|\mathcal{H}(L;f;\overline g)|=|\mathcal{H}(L;g;\overline f)|$.
    Clearly, if $\{f,g\}$ is a $2$-circuit of $L$, then $f$ and $g$ are clones in $L$. 
    Thus, we may assume that no 2-circuit of $L$ contains $g$.

    For $J\in \mathcal{H}(L;g;\overline f)$, let $B_J$ be an arbitrarily chosen basis of $J$ containing $g$. 
    Then $\cl_{L}(B_J-g)$ is a rank-$(r-2)$ flat of $L$. 
    Let $\cl_{L}(B_J-g)\cup f=F$. 
    Then $r(F)=r-1$ otherwise $f\in \cl_{L}(B_J-g)$, so $g\in \cl_{L}(B_J-g)$, a contradiction. 
    Consider $\cl_{L}(F)$. 
    Assume $f$ is not a coloop of $\cl_{L}(F)$. 
    Then $\cl_{L}(F)$ contains a circuit $C$ containing $f$. 
    Since $f$ is freer than $g$, we see that $g\in \cl_{L}(F)$. 
    Then $\cl_{L}(F)\supseteq B_J$. 
    Thus $\cl_{L}(F)\supseteq J$. 
    As $r(\cl_{L}(F))=r(J)$, we deduce that $\cl_{L}(F)=J$. 
    But $f\not\in J$, a contradiction. 
    Thus $f$ is a coloop of $\cl_{L}(F)$, so $\cl_{L}(B_J-g)\cup f\in \mathcal{H}(L;f;\overline g)$. 
    Let $\psi(J)=\cl_{L}(B_J-g)\cup f$. 
    Note that $\psi$ depends upon the choices made for the bases $B_J$. Moreover, $\psi$ maps $\mathcal{H}(L;g;\overline f)$ to $\mathcal{H}(L;f;\overline g )$.  

    \begin{sublemma}
       $\psi$ is bijective. 
    \end{sublemma}

    To see that $\psi$ is injective, suppose that, for distinct members $J_1$ and $J_2$ of $\mathcal{H}(L;g;\overline f)$, the hyperplanes $\psi(J_1)$ and $\psi(J_2)$ are equal. 
    Then $\cl_{L}(B_{J_1}-g)=\cl_{L}(B_{J_2}-g)$. 
    Now the rank-$(r-2)$ flat $J_1\cap J_2$ contains $g$ and so contains the rank-$(r-1)$ set $B_{J_1}$, a contradiction. 
    Since $|\mathcal{H}(L;g;\overline f)|=|\mathcal{H}(L;f;\overline g)|$ and $\psi$ is injective, we conclude that (12.1) holds.

    \begin{sublemma}
        $g$ is a coloop of $L|J$ for every $J\in \mathcal{H}(L;g;\overline f)$.  
    \end{sublemma}

    Suppose $g$ is not a coloop of $L|J$. 
    Then $\cl_{L}(B_J-g)$ is a subset of $J$ avoiding $g$. 
    As $g$ is not a coloop, there is an element $h$ of $J-\cl_{L}(B_J-g)-g$. 
    Since no 2-circuit of $L$ contains $g$, the elements $g$ and $h$ are not parallel. 
    Thus $\{g,h\}$ is independent. 
    Extend $\{g,h\}$ to a basis $B_J'$ of $L|J$. 
    Then $\cl_{L}(B_J'-g)\neq\cl_{L}(B_J-g)$ because $h\in \cl_{L}(B_J'-g)-\cl_{L}(B_J-g)$. 
    Thus $\cl_{L}(B_J'-g)\cup f$ is a member of $\mathcal{H}(L;f;\overline g)$ that is not in the set $\psi(\mathcal{H}(L;g;\overline f))$.
    As $|\mathcal{H}(L;g;\overline f)|=|\mathcal{H}(L;f;\overline g)|$ and $\psi$ is a bijection, this is a contradiction. Thus (12.2) holds. 

    \smallskip

    Suppose $g$ is not freer than $f$ in $L$.
    Then $L$ has a cyclic flat $K$ containing $g$ and avoiding $f$. 
    Take a basis $B$ for $K$ and consider $B \cup f$. 
    Extend $B\cup f$ to get a basis $B_{L}$ for $L$. 
    Then $\cl_L(B_{L}-f)$ is a hyperplane of $L$ containing $g$ and avoiding $f$. 
    Moreover, since $K\subseteq \cl_L(B_{L}-f)$, the hyperplane $\cl_L(B_{L}-f)$ has a circuit containing $g$; that is, $g$ is not a coloop of $\cl_L(B_{L}-f)$, a contradiction to (12.2).
       \end{proof} 

       The next corollary is obtained by applying Proposition \ref{hypclones} to $M^*$. 

\begin{corollary}\label{ccteq}
    Let $f$ be freer than $g$ in $M$. 
    Let $N$ be obtained from $M$ by contracting every element of $E(M)-\{f,g\}$ that is in series with $f$. 
    Then $\gamma(f;M)=\gamma(g;M)$ if and only if $f$ and $g$ are clones in $N$.     
\end{corollary}
\begin{proof}
    Clearly $\gamma(f;N)=\gamma(g;N)$ if $f$ and $g$ are clones in $N$.  
    Hence $\gamma(f;M)=\gamma(g;M)$. 
    To prove the converse, suppose that $\gamma(f;M)=\gamma(g;M)$. 
    Then $\gamma(f;N)=\gamma(g;N)$. 
    If $f$ and $g$ are in series in $N$, then $f$ and $g$ are clones in $N$.  
    Thus, we will assume that $f$ and $g$ are not a series pair in $N$.  
    
    Let $\mathcal{C}(M;f,\overline g)$ be the set of circuits of $M$ containing $f$ and avoiding $g$. 
    Then $|\mathcal{C}(N;f;\overline g)|=|\mathcal{C}(N;g;\overline f)|$. 
    By duality, $|\mathcal{C}(N;f,\overline g)|=|\mathcal{H}(N^*;g;\overline f)|$.
    Therefore $|\mathcal{H}(N^*;g;\overline f)|=|\mathcal{H}(N^*;f;\overline g)|$, and, consequently, $h(g;M^*)=h(f;M^*)$. 
    Since $g$ is freer than $f$ in $M^*$, by Proposition \ref{hypclones}, we see that $f$ and $g$ are clones in $M^*\del X$ where $X$ is the set of elements of $E(M)-\{f,g\}$ that are parallel to $f$ in $M^*$.
    It follows that $f$ and $g$ are clones in $N$. 
\end{proof}

Based on Theorem \ref{beq}, Proposition \ref{hypclones}, and Corollary \ref{ccteq}, one may guess that, when $f$ is freer than $g$ in $M$ and $\gamma'(f;M) = \gamma'(g;M)$, the elements $f$ and $g$ must be clones in $M$ when $f$ is not in any $2$-cocircuit of $M$. 
To see that this is not so, let $M$ be the rank-$5$ matroid that is obtained by taking the 2-sum across a common basepoint $p$ of two 4-point lines $M_2$ and $M_3$ and of a 6-element rank-$3$ matroid $M_1$ that has $\{g,a,b\}$ as its only non-spanning circuit and has $f$, $p$, and $c$ as free elements. 
Then $f$ is freer than $g$ in $M$ and $\gamma'(f;M)=0=\gamma'(g;M)$. But $f$ and $g$ are not clones in the cosimple matroid $M$.

The \textit{truncation} of $M$, which we will denote $\tau(M)$, is the matroid obtained from $M$ by taking the free extension $M+_{E(M)}e$ of $M$ by $e$ and then contracting the free element $e$. In particular, when $r(M)>0$, the independent sets of $\tau(M)$ are the independent sets of $M$ with at most $r(M)-1$ elements. Note that we use $\tau(M)$ rather than the more standard $T(M)$ to denote truncation in order to avoid confusion with the Tutte polynomial of $M$. 

\begin{corollary}\label{trunc}
    Let $r(M)=r\geq 2$. If $f$ is freer than $g$ in $M$, then $W_{r-2}(f;M)=W_{r-2}(g;M)$ if and only if $f$ and $g$ are clones in the matroid obtained from $\tau(M)$ by deleting every element of $E(M)-\{f,g\}$ that is parallel to $g$.  
\end{corollary}
    \begin{proof}
        Let $F$ be a rank-$(r-2)$ flat of $M$. 
        Then $F$ is a rank-$(r-2)$ flat of $M+_{E(M)}e$ avoiding $e$.
        Hence $F$ is a hyperplane of $\tau(M)$.
        Therefore, the rank-$(r-2)$ flats of $M$ containing an element $x$ are precisely the hyperplanes of $\tau(M)$ containing $x$, that is, $W_{r-2}(f;M)=W_{r-2}(g;M)$ if and only if $h(f;\tau(M))=h(g;\tau(M))$. 
        As $f$ is freer than $g$ in $\tau(M)$, the result follows immediately from Proposition \ref{hypclones}. 
    \end{proof}

    The following result generalizes Corollary \ref{trunc} to the \textit{i-th truncation} $\tau^i(M)$ of $M$, defined recursively by \mbox{$\tau^i(M)=\tau(\tau^{i-1}(M))$} where $\tau^0(M)=M$.     

    \begin{corollary}
        Suppose $f$ is freer than $g$ in $M$ and $r(M)=r\geq 2$. Then $W_{k}(f;M)=W_k(g;M)$ for some $k$ with $1\leq k\leq r-1$ if and only if $f$ and $g$ are clones in the matroid obtained from $\tau^{r-k-1}(M)$ by deleting every element of $E(M)-\{f,g\}$ that is parallel to $g$.
    \end{corollary}

    \begin{proof}
        This follows immediately from Corollary \ref{trunc}. 
    \end{proof}

\end{document}